\documentclass[conference]{IEEEtran}
\usepackage{graphicx}
\usepackage{amsmath,amsthm,amssymb,mathrsfs,amsfonts,amsopn,bbm}
\usepackage{enumerate}
\usepackage{multicol}

\def\ff{{\mathcal F}}

\def\ffi{\varphi}
\def\eps{\varepsilon}

\DeclareMathOperator{\supp}{supp}

\def\C{{\mathbb{C}}}

\def\R{{\mathbb{R}}}

\newtheorem{theorem}{Theorem}[section]
\newtheorem{lemma}[theorem]{Lemma}
\newtheorem{question}[theorem]{Question}
\newtheorem{proposition}[theorem]{Proposition}

\newtheorem{corollary}[theorem]{Corollary}

\theoremstyle{definition}
\newtheorem{definition}{Definition}

\newtheorem{notation}[theorem]{Notation}

\begin{document}

\title{Sign retrieval in spaces of variable bandwidth}
\author{Philippe Jaming$^1$ \& Rolando Perez III$^{2}$\\
{\small $^1$Univ. Bordeaux, CNRS, Bordeaux INP, IMB, UMR 5251,  F-33400, Talence, France}\\
{\small $^2$Institute of Mathematics, University of the Philippines Diliman, 1101 Quezon City, Philippines}}
\maketitle

\begin{abstract}
The aim of this paper is to get a deeper understanding of the
spaces of variable bandwidth introduced by Gr\"ochenig and Klotz
({\it What is variable bandwidth?} Comm. Pure Appl. Math., {\bf 70}  (2017), 2039--2083).
In particular, we show that when the variation of the bandwidth is modeled
by a step function with a finite number of jumps, then, the sign retrieval
principle applies.
\end{abstract}

\begin{IEEEkeywords} Spaces of variable bandwidth, sign retrieval, phase retrieval
\end{IEEEkeywords} 

\section{Introduction}
The phase retrieval problem is an ubiquitous family of problems in the applied sciences
when one wants to reconstruct a signal $f$ from its modulus $|f|$ and some a priori
knowledge on $f$ that is usually expressed through the fact that $f$
belongs to some function space $\ff$. In other words, we are asking whether $f,g\in\ff$ with
$|f|=|g|$ implies $f=cg$ where $c$ is a global phase factor, that is a complex number of modulus 1. When one further restricts the signals $f,g$ to be real valued, then 
we ask whether $f=g$ or $f=-g$, {\it i.e.}, $f$ and $g$ are the same up to a global sign.

Our general aim is to investigate which properties of the function space lead to the
sign retrieval property. To explain the general aim, let us focus first on the Paley-Wiener
spaces. Fix $c>0$ and recall that the classical Paley-Wiener space $PW_c(\R)$ is the set of all $L^2$ functions whose Fourier transforms are supported on the interval $[-c,c]$, that is,
$$
PW_c(\R)=\{f\in L^2(\R)\,:\ \supp\widehat f\subset [-c,c]\}.
$$
Here, we use the normalized definition of the Fourier transform given by
$$
\widehat f(\xi)=\mathcal Ff(\xi)=\dfrac{1}{\sqrt{2\pi}}\int_\R f(x)e^{-ix\xi}\,\textnormal{d}x.
$$
Functions in $PW_c(\R)$ are called band-limited functions and $c$ is the bandwidth.
The sign-retrieval problem is easy to solve in this space\,:

\begin{lemma}[Paley-Wiener sign retrieval]\label{lem:srpw}
Let $I\subset\R$ be an interval, and let $f,g\in PW_c(\R)$ are real valued (on $\R$)
and such that $|f(x)|=|g(x)|$ for $x\in I$. Then either $f=g$ or $f=-g$.
\end{lemma}

Indeed, it is known that each function in $PW_c(\R)$ extends to an entire function
and this is the only property needed\,: If $f,g$ are real valued on $\R$ and satisfy 
$|f(x)|=|g(x)|$ for $x\in I$
then there is a set $E\subset I$
such that $f=g$ on $E$ and $f=-g$ on $I\setminus E$. Now at least one of 
$E$ or $I\setminus E$ has an accumulation point, say $E$.
As $f$ and $g$ are entire functions, $f=g$ on $E$ implies $f=g$ on $\C$.

In particular, the sign retrieval is valid in any space of entire functions
(or directly related to such a space) like {\it e.g.} de Branges spaces,
or time-warping variable bandwidth spaces, {\it i.e.}, spaces of functions of the form
$f\circ \gamma$ where $f\in PW_c(\R)$ and $\gamma$ is a strictly increasing homeomorphism of $\R$
as introduced in \cite{Ho} ({\it see} also \cite{GK2017} and further references therein).

Further, sampling theorems are a key feature of the Paley-Wiener spaces and
the next natural question is of course to know for which (discrete) sets $\Lambda\subset\R$,
$|f(x)|=|g(x)|$ for $x\in \Lambda$ implies that $f=g$ or $f=-g$. This has been first investigated
by Thakur in \cite{Ta}, further investigated in shift invariant spaces \cite{Gr,Ro}.
We also refer to \cite{GLR} for deep results in this direction for the
Fock space. 

Our general aim is to be able to extend such results to spaces of variable bandwidth.
This is a natural concept in the physical sciences for which there is at this stage
no clear mathematical formulation. The paper \cite{GK2017} by Gr\"ochenig and Klotz
offers an overview of various attempts as well as a first definition of a new family
of spaces of variable bandwidth. For those spaces, the bandwidth is described by two parameters,
a global bandwidth $\Omega$ and a local variation $p(x)$ so that, at some point $x$,
the bandwidth is $c(x)=\Omega/\sqrt{p(x)}$. 
We will denote those spaces as $PW_\Omega(A_p)$.\footnote{We have decided to adopt a slightly different normalization, this space would be $PW_{\Omega^2}(A_p)$ in \cite{GK2017}.}
This note stems from our will to better understand
those spaces and to see whether they are adapted to phase retrieval or at least
to sign retrieval, if possible in its sampled form. We conjecture
that this is possible, at least when the local variation $p(x)$ is a step function
with finitely many jumps. The reason we think this is possible are the following\,:

-- $PW_\Omega(A_p)$ is a reproducing kernel Hilbert space consisting of
continuous functions. Moreover, on an interval on which $p$ is constant, the functions
in $PW_\Omega(A_p)$ are restrictions of functions in $PW_c(\mathbb R)$ to this interval.

-- There are good sampling formulas in $PW_\Omega(A_p)$.

-- Sign retrieval is possible in $PW_\Omega(A_p)$. This result is new and is the main result of this note. We offer two proofs, one in the so-called toy model when $p$ is a step function with a single jump and that is based on the reproducing kernel. The second one is for general step functions.

The remaining of the paper is organized as follows\,: The next section is devoted to
recalling the main facts from \cite{GK2017}, reformulating the definition
of $PW_\Omega(A_p)$ when $p$ is a step function with finitely many jumps and then proving
the sign retrieval property. The last section is devoted to giving a second proof of this 
property in the case of a single jump, based this time on the reproducing kernel.

\section{Functions of Variable Bandwidth}
The starting point of Gr\"ochenig and Klotz \cite{GK2017} is as follows\,:
Band-limited functions are contained in a spectral subspace of the differential operator $-D^2=-\frac{d^2}{dx^2}$ since this operator is diagonalized by the Fourier transform $\mathcal F$ with $-\mathcal F D^2\mathcal F^{-1}f(\xi)=\xi^2\widehat f(\xi)$. The idea in \cite{GK2017}
is to replace $-D^2$ by the Sturm-Liouville operator $\tau_p$ given by
$$
\tau_p f(x)=-(pf')'(x),\qquad x\in\R
$$
where $p>0$ is a bandwidth-parametrizing function. 
For almost every $x\in\R$ and $z\in\C$, the Sturm-Liouville equation corresponding to $\tau_p$ is given by $(\tau_p-z)f=0$ on $\R$.
%

We will here focus on a particular case that was investigated in more depth
by Celiz, Gr\"ochenig and Klotz \cite{CGK}\,:
\begin{notation}
We will say that $p\in PC_N$ if $p$ is a step function with $N$ jumps.
We will further write $x_0=-\infty<x_1<\cdots<x_N<x_{N+1}=+\infty$,
($x_1,\ldots,x_N$ are the jump points), $I_j:=(x_j,x_{j+1})$, $j=0,\ldots,N$,
$p_0,\ldots,p_{N}>0$ and on each $I_j$, $p(x)=p_j$.
We will also write $q_j=p_j^{-1/2}$.
\end{notation}
 
 To a given $\tau_p$ corresponds the maximal operator $(A_p,\mathcal D(A_p))$ given by
 \begin{eqnarray*}
 \mathcal D(A_p)&=&\{f\in L^2(\R)\:\ f,pf'\in AC_{loc}(\R),~\tau_p\in L^2(\R)\}\\
 A_pf&=&\tau_p f,~f\in \mathcal D(A_p).
\end{eqnarray*}
where $AC_{loc}(\R)$ is the space of functions that locally are absolutely continuous (integrals of their distributional derivative).
 

Now, for a Sturm-Liouville operator $\tau_p$ associated to $p\in PC_N$, a solution $\phi_z$ of $(\tau_p-z)f=0$ lies left in $L^2(\R)$ if $\phi_z\in L^2(I_0)$, and lies right if $\phi_z\in L^2(I_{N})$. Then, for every $z\notin\R$ there are two unique solutions of $(\tau_p-z)f=0$ up to a multiplicative constant, one of which lies left and one of which lies right in $L^2(\R)$ and the corresponding maximal operator $A_p$ is self-adjoint. We will need the following\,:

\begin{theorem}[\cite{GK2017}, Theorem 2.3]
	Let $p\in PC_N$ and $A_p$ be the self-adjoint realization of $\tau_p$. 
	If $\Phi(\lambda,x)=(\phi_+(\lambda,x),\phi_-(\lambda,x))$, for $\lambda,x\in\R$, is a fundamental system of solutions of $(\tau_p-\lambda)\phi=0$ that continuously depends on $\lambda$, then there exists a $2\times2$ matrix measure $\mu$ such that the operator
	$\mathcal F_{A_p}\,:\ L^2(\R)\longrightarrow L^2(\R,\textnormal{d}\mu)$
	\begin{equation*}
		\label{eq:spectralFT}
		 \mathcal F_{A_p}f(\lambda)=\int_\R f(x)\overline{\Phi(\lambda,x)}\,\textnormal{d}x
	\end{equation*}
	is unitary and diagonalizes $A_p$, i.e.,
	$$
	\mathcal F_{A_p} A_p \mathcal{F}_{A_p}^{-1} G(\lambda)=\lambda G(\lambda),\qquad G\in L^2(\R,\textnormal{d}\mu).
	$$
	Then $\mathcal F_{A_p}$ is called the spectral (Fourier) transform of $A_p$
	and its inverse is given as
	$$
	\mathcal F_{A_p}^{-1} G(x)=\int_\R G(\lambda)\cdot \Phi(\lambda,x)\,\textnormal{d}\mu(\lambda),\qquad G\in L^2(\R,\textnormal{d}\mu).
	$$
\end{theorem}

Now, when $p\in PC_N$, this can be made (almost) explicit.
A fundamental system of solutions is given by
$\Phi_\pm(z,x)=\tilde \Phi_\pm(\sqrt{z},x)$ where by convention,
when $z=re^{i\theta}$, $r>0$ and $-\pi<\theta<\pi$, $\sqrt{z}=\sqrt{r}e^{i\theta/2}$,
and, for $x\in I_j$,
\begin{equation}
\label{eq:expphi}
\tilde\Phi_\pm(\zeta,x)=a_j^\pm(\zeta)e^{iq_j\zeta x}+b_j^\pm(\zeta)e^{-iq_j\zeta x},
\end{equation}
and the coefficients $A_j^\pm:=\begin{pmatrix}a_j^\pm\\ b_j^\pm\end{pmatrix}$ are given by the induction formula $A_0^-=\begin{pmatrix}0\\ 1\end{pmatrix}$, $A_N^+=\begin{pmatrix}1\\ 0\end{pmatrix}$ and $A_j^\pm=\dfrac{1}{2}L_j(\zeta)A_{j-1}^\pm$
where $L_j(\zeta)$ is the invertible matrix
$$
\begin{pmatrix}
\left(1+\frac{q_j}{q_{j-1}}\right)e^{ix_j(q_j-q_{j-1})\zeta}&
\left(1-\frac{q_j}{q_{j-1}}\right)e^{-ix_j(q_j+q_{j-1})\zeta}\\
\left(1-\frac{q_j}{q_{j-1}}\right)e^{ix_j(q_j+q_{j-1})\zeta}&
\left(1+\frac{q_j}{q_{j-1}}\right)e^{-ix_j(q_j-q_{j-1})\zeta}
\end{pmatrix}.
$$
It is then easy to see that $a_j^\pm,b_j^\pm$ are almost periodic trigonometric polynomials,
in particular, they are bounded. 
An other key fact for us is that 
\begin{equation}
\label{eq:det}
\det\begin{pmatrix}a_j^+(\zeta)&a_j^-(\zeta)\\ b_j^+(\zeta)&b_j^-(\zeta)\end{pmatrix}\not=0.
\end{equation}
The spectral measure is also explicitly given by
$$
\mbox{d}\mu(\lambda)=\frac{1}{4\pi p_n}\begin{pmatrix}\dfrac{1}{q_0}&0\\ 0&\dfrac{1}{q_n}\end{pmatrix}\dfrac{\mbox{d}\lambda}{|b_n^-(\sqrt{\lambda})|^2\sqrt{\lambda}}.
$$

We can now define the following spectral projections\,: For $\Lambda$ a Borel subset of $\R^+$ and $f\in L^2$,
\begin{eqnarray*}
\chi_{\Lambda} (A_p)f(x)
&=&\int_\Lambda \mathcal F(A_p)f(\lambda)\cdot \Phi(\lambda,x)\,\mbox{d}\mu(\lambda)\\
&=&\mathcal F_{A_p}^{-1}(\chi_{\Lambda}\mathcal F_{A_p}f)(x),\qquad x\in\R.
\end{eqnarray*}


We now define the Paley-Wiener space of variable bandwidth functions\,:

\begin{definition}[\cite{GK2017}]
	Let $p\in PC_N$ and $A_p$ be the self-adjoint realization of $\tau_p$.
	For $\Omega>0$, the Paley-Wiener space of variable bandwidth functions, denoted by $PW_\Omega(A_p)$, is the range of the spectral projection $\chi_{[0,\sqrt{\Omega}]} (A_p)$, i.e.,
	\begin{eqnarray*}
	PW_\Omega (A_p)&=&\chi_{[0,\sqrt{\Omega}]} (A_p)(L^2(\R))\\
	&=&\{f\in L^2(\R)\,:\ f=\chi_{[0,\sqrt{\Omega}]} (A_p)f\}.
	\end{eqnarray*}
\end{definition}

Note that when $p\equiv1$, $PW_\Omega (A_p)=PW_{\Omega}(\R)$.
We also recall a characterization for the spaces $PW_\Omega(A_p)$ akin to the one from classical Paley-Wiener spaces\,:

\begin{theorem}[\cite{GK2017}, Proposition 3.2]
Let $p\in PC_N$, $A_p$ be the self-adjoint realization of $\tau_p$ and $\Omega>0$.
Let $\mu$ be the spectral measure of $A_p$ with corresponding spectral transform $\mathcal F_{A_p}$. Then the following are equivalent\,:
	\begin{enumerate}
		\item $f\in PW_\Omega(A_p)$,
		\item $\supp \mathcal F_{A_p}f\subseteq [0,\sqrt{\Omega}]$,
		\item there exists a function $F\in L^2([0,\sqrt{\Omega}],\textnormal{d}\mu)$ such that for almost every $x\in\R$,
		\begin{equation}
		\label{eq:PWexp}
		f(x)=\int_0^{\sqrt{\Omega}} F(\lambda)\cdot\Phi(\lambda,x)\,d\mu(\lambda).
		\end{equation}
More explicitly, there exist $F_+,F_-$ such that 
\begin{align*}
f(x)=& \int_0^{\sqrt{\Omega}}
\left(\frac{F_+(\lambda)\tilde\phi_+(\sqrt{\lambda},x)}{q_0}+
		\frac{F_-(\lambda)\tilde\phi_-(\sqrt{\lambda},x)}{q_n}\right)\notag\\
		&\qquad\frac{\textnormal{d}\lambda}{4\pi p_n|b_n^-(\sqrt{\lambda})|^2\sqrt{\lambda}}.\notag\\
\end{align*}
	\end{enumerate}
	\end{theorem}
\vspace*{-2mm}
Using the explicit expressions given above, this can be written as
\begin{equation*}
f(x)=
\int_0^{\Omega}\bigl(G_-(\zeta)\tilde\phi_-(\zeta,x)
+G_+(\zeta)\tilde\phi_+(\zeta,x)\bigr)\,\mbox{d}\zeta
\end{equation*}
after a change of variable $\zeta=\sqrt{\lambda}$
and change of unknown function, $G_-(\zeta)=\dfrac{F_-(\zeta^2)}{2\pi p_nq_n|b_n^-(\zeta)|^2}$, and a similar expression for $G_+$ so that $G_\pm\in L^2(0,\Omega)$.
Using \eqref{eq:expphi}, this can further be written as
$$
f(x)=\int_{-\Omega}^{\Omega}G_j(\zeta)e^{-iq_j\zeta x}\,\mbox{d}\zeta,\quad x\in I_j
$$
where
$$
G_j(\zeta)=\begin{cases}
b_j^-(\zeta)G_-(\zeta)+b_j^+(\zeta)G_+(\zeta),
&\mbox{for }\zeta>0\\
a_j^-(-\zeta)G_-(-\zeta)+a_j^+(-\zeta)G_+(-\zeta),
&\mbox{for }\zeta<0\end{cases}.
$$
As $G_\pm\in L^2(0,\Omega)$ and $a_j^\pm,b_j^\pm$ are bounded, $G_j\in L^2(-\Omega,\Omega)$.
We thus obtain the following, which was already partially  in
\cite[Proposition 3.3(i), Proposition 3.5]{GK2017}, 

\begin{corollary}
Let $p\in PC_N$, $A_p$ be the self-adjoint realization of $\tau_p$ and $\Omega>0$.
Then $PW_{\Omega}(A_p)$ is a closed subspace of $L^2(\R)$ consisting of continuous functions.
Moreover, on each $I_j$, $f$ is the restriction to $I_j$ of a function
in $PW_{\Omega q_j}(\R)$.
\end{corollary}

The main difference is that in \cite{GK2017}, the authors prove that
$f$ is locally in the larger Bernstein space. On the other hand, their result is valid for more general $p$. This corollary also leads to the following question\,:

\begin{question}
Let $p\in PC_N$ and, for each $j$, let $f_j$ be the restriction to $I_j$ of 
a function in $PW_{\Omega q_j}(\R)$. Under which condition is the function $\displaystyle f=\sum_{j=0}^Nf_j\mathbbm{1}_{I_j}\in PW_{\Omega}(A_p)$?
\end{question}

Of course, $f$ has to be continuous, but this is not enough. So far, we have only been able
to answer this question in the single jump case. This will be done in the next section
and also leads to another proof of the sign retrieval question in this case. It is likely
that this would also be a crucial step for phase retrieval in variable band width spaces.

This raises also a second question\,:

\begin{question}
Let $p,\tilde p\in PC_N$ be such that $p_1\geq p_2$. Is $PW_{\Omega}(A_p)\subset PW_{\Omega}(A_{\tilde p})$?
\end{question}

From the formula, this is clear when $\tilde p=\alpha p$, but not in the general case.

\smallskip

Our computation also shows that $PW_\Omega(A_p)$ is obtained as follows\,:
take a pair of function $G_-,G_+\in L^2(0,\Omega)$. For each $\zeta$,
and each $j$, construct $G_j$ with the above formula. Then on $I_j$, take the (properly scaled) Fourier transform of $G_j$. Now this can be reverted.
If $f\in PW_\Omega(A_p)$ is given, then its values on any $I_j$
uniquely determine $G_+,G_-$ as follows\,: As $f$ is the restriction of a function $f_j\in PW_{c_j}$,
we may (in theory) obtain $f_j$ on $\R$ from its restriction to $I_j$.
Consider its Fourier transform $\widehat{f_j}$. To obtain $G_-,G_+$, it remains to solve the system
$$
\left\{\begin{matrix}
b_j^-(\zeta)G_-(\zeta)&+&b_j^+(\zeta)G_+(\zeta)&=&\widehat{f_j}(\zeta)\\
a_j^-(\zeta)G_-(\zeta)&+&a_j^+(\zeta)G_+(\zeta)&=&\widehat{f_j}(-\zeta)
\end{matrix}\right.
$$
for $\zeta>0$. This system has non-zero determinant, so its solution is unique.

We are now ready to prove the main result\,:

\begin{theorem}
Let $p\in PC_N$, $A_p$ be the self-adjoint realization of $\tau_p$ and $\Omega>0$.
Let $f,g\in PW_\Omega(A_p)$ be real valued and such that $|f|=|g|$.
Then $f=g$ or $f=-g$.
\end{theorem}

\begin{proof}
We first consider $|f|=|g|$ on each of the intervals $I_j$. As $f$ (resp. $g$)
is the restriction of a function $f_j$ (resp. $g_j$) in $PW_{c_j}(\R)$,
according to Lemma \ref{lem:srpw}, there exists $\eps_j\in\{\pm 1\}$
such that $g_j=\eps_j f_j$ on $I_j$. 
Next, as just explained,
$f$ (resp. $g$) stems from a pair of functions $G_\pm^f$ (resp. $G_\pm^g$) and they are 
determined by the system
$$
\left\{\begin{matrix}
b_j^-(\zeta)G_-^f(\zeta)&+&b_j^+(\zeta)G_+^f(\zeta)&=&\widehat{f_j}(\zeta)\\
a_j^-(\zeta)G_-^f(\zeta)&+&a_j^+(\zeta)G_+^f(\zeta)&=&\widehat{f_j}(-\zeta)
\end{matrix}\right.
$$
and
$$
\left\{\begin{matrix}
b_j^-(\zeta)G_-^g(\zeta)&+&b_j^+(\zeta)G_+^g(\zeta)&=&\eps_j\widehat{f_j}(\zeta)\\
a_j^-(\zeta)G_-^g(\zeta)&+&a_j^+(\zeta)G_+^g(\zeta)&=&\eps_j\widehat{f_j}(-\zeta)
\end{matrix}\right..
$$
But these systems are invertible, so that $(G_-^g,G_+^g)=\eps_j(G_-^f,G_+^f)$.
There are three cases

-- either $\eps_j=1$ for all $j$ and then $(G_-^g,G_+^g)=(G_-^f,G_+^f)$
and finally $g=f$;

-- or $\eps_j=-1$ for all $j$ and then $g=-f$;

-- or there are $k,\ell$ such that $\eps_k=1$ and $\eps_\ell=-1$. But then
we simultaneously have $(G_-^g,G_+^g)=(G_-^f,G_+^f)$ and $(G_-^g,G_+^g)=-(G_-^f,G_+^f)$
so that $G_\pm^g=G_\pm^f=0$ and finally $g=f=0$.
\end{proof}

\section{A second proof for the toy Example} 
The toy example is simple case of piecewise constant functions with a single jump,
which one can put at $0$\,:
$$
p(x)=
\begin{cases}
	p_-,&x\leq 0\\
	p_+&x>0
\end{cases}
$$
with $p_-,p_+>0$. 

In this case, we will give a second proof of the main result based on the fact
that the space $PW_{\Omega}(A_p)$ is a reproducing kernel Hilbert space \cite[Proposition 3.3(ii)]{GK2017} with kernel
$$
k_\Omega(x,y)=\int_0^\Omega\overline{\Phi(\lambda,x)}\Phi(\lambda,y)\,\mbox{d}\mu(\lambda).
$$
This kernel is also the kernel of the projection from $L^2(\R)$ onto $PW_{\Omega}(A_p)$.
The main issue is that this kernel is far from being explicit excepted in the toy model.
Even in the case of two jumps, the reproducing kernel is essentially untractable,
as can be seen from the one full-page formula in \cite{CGK}.

To come back to the toy model, the reproducing kernel for the space $PW_{[0,\Omega]}(A_p)$ is given by
\begin{enumerate}
\renewcommand{\theenumi}{\roman{enumi}}
\item when $x,y\leq 0$
\begin{multline*}
k(x,y)=\dfrac{c_-}{\pi}\operatorname{sinc}(c_-(x-y))\\
 -(\rho_+-\rho_-)\dfrac{c_-}{\pi}\operatorname{sinc}(c_-(x+y));
 \end{multline*}
\item when $x\leq 0,~y\geq 0$
$$
k(x,y)=\dfrac{2c_+\rho_+}{\pi}\operatorname{sinc}(c_-x-c_+y);
$$
\item when $x,y>0$,
\begin{multline*}
k(x,y)=\dfrac{c_+}{\pi}\operatorname{sinc}(c_+(x-y))\\
 +(\rho_+-\rho_-)\dfrac{c_+}{\pi}\operatorname{sinc}(c_+(x+y));
 \end{multline*}
\item when $x>0,~y\leq 0$
$$
k(x,y)=\dfrac{2c_-\rho_-}{\pi}\operatorname{sinc}(c_+x-c_-y)
$$
\end{enumerate}
where 
$$
	c_{\pm}=\sqrt{\dfrac{\Omega}{p_{\pm}}}~\text{and}~\rho_{\pm}=\dfrac{\sqrt{p_{\pm}}}{\sqrt{p_+}+\sqrt{p_-}}.
$$
Note that 
if $p_+=p_-:=p$ and $c=\sqrt{\dfrac{\Omega}{p}}$ then 
$k(x,y)=\dfrac{c}{\pi}\operatorname{sinc}\bigl(c(x-y)\bigr)$ and $PW_{\Omega}(A_p)=PW_{c}(\R)$.

From these formulas, it is possible to relate $f\in PW_{[0,\Omega]}(A_p)$
to the usual Paley-Wiener spaces. To do so, recall that the orthogonal projection of a function $f\in L^2(\R)$ onto $PW_c(\R)$ is given by
$$
\Pi_c f(x)=\dfrac{c}{\pi}\int_\R \operatorname{sinc}(c(x-y))f(y)\,dy,~x\in\R.
$$
For $a>0$, let $\delta_a$ be the  dilation of a function $f\in L^2(\R)$ given by $\delta_a f(x)= \sqrt{a}f(ax)$ so that $\|\delta_af\|_2=\|f\|_2$. It is easy to see that $\widehat{\delta_a f}=\delta_{1/a}\widehat f$, thus $f\in PW_c(\R)$ if and only if $\delta_af\in PW_{ac}(\R)$,
and that $\delta_a\Pi_c[\delta_{1/a}f]=\Pi_{ac}f$.

Now, for $f\in PW_{[0,\Omega]}(A_p)$, write $f_\pm=f\mathbbm{1}_{\R^\pm}$.
Then, for $x\in\R^-$, 
\begin{align*}
f(x)&=\dfrac{1}{\pi}\int_{-\infty}^0 c_-\operatorname{sinc}\left(c_-(x-y)\right)f_-(y)\,\mbox{d}y\\ 
&\quad-\dfrac{\rho_+-\rho_-}{\pi}\int_0^{\infty}c_-\operatorname{sinc}(c_-(x-y))f_-(-y)\,\mbox{d}y\\
&\qquad-\dfrac{2\rho_+}{\pi}\int_0^{\infty}c_+\operatorname{sinc}\left(c_+\left(\frac{c_-}{c_+}x-y\right)\right)f_+(y)\,\mbox{d}y\\
&=\Pi_{c_-}\left[f_-\cdot \mathbbm 1_{\R_-}
-(\rho_+-\rho_-)\check f_-\cdot \mathbbm 1_{\R_+}\right](x)\\
&\quad+2\rho_+\sqrt{\dfrac{c_+}{c_-}}\delta_{c_-/c_+}\Pi_{c_+}[f_+\cdot \mathbbm 1_{\R_+}](x),
\end{align*}
where $\check \varphi(x)=\varphi(-x)$. Using the scaling property, 
we obtain
\begin{align}
f_-&=\mathbbm{1}_{\R_-}\Pi_{c_-}\Bigl[f_-\cdot \mathbbm 1_{\R_-}-(\rho_+-\rho_-)\check f_-\cdot \mathbbm 1_{\R_+}\notag\\
&\quad+2\rho_+\sqrt{\dfrac{c_+}{c_-}}(\delta_{c_-/c_+}f_+))\cdot \mathbbm 1_{\R_+}\Bigr].
\label{eq:f-}
\end{align}
Similarly, the positive part is given by
\begin{align}
f_+&=\mathbbm{1}_{\R_+}\Pi_{c_+}\Bigl[f_+\cdot \mathbbm 1_{\R_+}+(\rho_+-\rho_-)\check f_+\cdot \mathbbm 1_{\R_-}\notag\\
&\quad+2\rho_-\sqrt{\dfrac{c_-}{c_+}}(\delta_{c_+/c_-}f_-))\cdot \mathbbm 1_{\R_-}\Bigr].
\label{eq:f+}
\end{align}

\begin{proposition}[Sign Retrieval]
Assume that $p\in PC_1$.
Let $f,g\in PW_{[0,\Omega]}(A_p)$ be \emph{real-valued} and such that $|f|=|g|$ on $\R$, then
$f=g$ of $f=-g$.
\end{proposition}

\begin{proof}
%

Now, let $f,g\in PW_{[0,\Omega]}(A_p)$ be \emph{real valued} and such that $|f|=|g|$ on $\R$,
Write $f_\pm=f\mathbbm 1_{\R_\pm}$ and $g_\pm=g\mathbbm 1_{\R_\pm}$. Each of $f_\pm$ or $g_\pm$ extends
holomorphically to $\C$ so that $|f_\pm|=|g_\pm|$ implies that there are $\eps_-,\eps_+\in\{-1,1\}$
such that $f_\pm=\eps_\pm g_\pm$. We want to show that $\eps_-=\eps_+$.
Assume this is not the case. Without loss of generality, we may assume that $\eps_-=1$ and $\eps_+=-1$
that is $f_-=g_-$ and $f_+=-g_+$. Let us plug this into \eqref{eq:f-} to obtain
\begin{align*}
f_-&=\mathbbm{1}_{\R_-}\Pi_{c_-}\Bigl[f_-\cdot \mathbbm 1_{\R_-}-(\rho_+-\rho_-)\check f_-\cdot \mathbbm 1_{\R_+}\\
&\quad+2\rho_+\sqrt{\dfrac{c_+}{c_-}}(\delta_{c_-/c_+}f_+)\cdot \mathbbm 1_{\R_+}\Bigr]\\
=\,\,g_-&=\mathbbm{1}_{\R_-}\Pi_{c_-}\Bigl[f_-\cdot \mathbbm 1_{\R_-}-(\rho_+-\rho_-)\check f_-\cdot \mathbbm 1_{\R_+}\\
&\quad-2\rho_+\sqrt{\dfrac{c_+}{c_-}}(\delta_{c_-/c_+}f_+)\cdot \mathbbm 1_{\R_+}\Bigr].
\end{align*}
Comparing both expressions of $f_-$ we obtain that
$$
\mathbbm{1}_{\R_-}\Pi_{c_-}\left[(\delta_{c_-/c_+}f_+)\cdot \mathbbm 1_{\R_+}\right]=0.
$$
As $\Pi_{c_-}[\ffi]$ is holomorphic, if it vanishes on $\R_-$, it is zero everywhere thus
$$
\delta_{c_-/c_+}\Pi_{c_+}[f_+\cdot \mathbbm 1_{\R_+}]=\Pi_{c_-}\left[(\delta_{c_-/c_+}f_+)\cdot \mathbbm 1_{\R_+}\right]=0
$$
that is $\Pi_{c_+}[f_+\cdot \mathbbm 1_{\R_+}]=0$. 
This means that $\ffi:=f_+\cdot \mathbbm 1_{\R_+}$ has Fourier transform supported in $\R\setminus[-c_+,c_+]$. A version of the uncertainty principle ({\it see} \cite[p36]{HJ}) shows that this
can only happen if $\ffi=0$. 
Now that we now that $f_+=0$ thus $g_+=0$, it follows that
$$
f_-=g_-=\mathbbm{1}_{\R_-}\Pi_{c_-}\left[f_-\cdot \mathbbm 1_{\R_-}-(\rho_+-\rho_-)\check f_-\cdot \mathbbm 1_{\R_+}\right].
$$

Further, from the expression of $f_+$, we obtain that
\begin{eqnarray*}
0&=&2\rho_-\sqrt{\dfrac{c_-}{c_+}}\mathbbm{1}_{\R_+}\Pi_{c_+}\left[(\delta_{c_+/c_-}f_-)\cdot \mathbbm 1_{\R_-}\right]\\
&=&\rho_-\sqrt{\dfrac{c_-}{c_+}}\mathbbm{1}_{\R_+}\Pi_{c_-}\left[f_-\mathbbm 1_{\R_-}\right]
\end{eqnarray*}
that is, $\Pi_{c_-}\left[f_-\mathbbm 1_{\R_-}\right]$ vanishes on $\R^+$ which, as for $f_+$ implies that
$f_-\mathbbm 1_{\R_-}=0$ so that, finally, $f_+=f_-=0$ and thus $f=g=0$.
\end{proof}


\section{Some future directions of research}

As we mentioned in the introduction, it is natural to ask whether sign retrieval is possible
from sampled measurements. This is possible in the classical Paley-Wiener space \cite{Ta},
but also in certain shift-invariant spaces \cite{Gr,Ro}. For variable band-width spaces,
sampling theorems are one of the key questions addressed in \cite{Gr,CGK}. At this stage, we have unfortunately 
been unable to give an answer to the following question:

\begin{question} 
Fix $p\in PC_N$. Under which conditions on a discrete set $\Lambda\subset\R$
is it true that $f,g\in PW(A_p)$ real valued with $|f(\lambda)|=|g(\lambda)|$ for every $\lambda\in\Lambda$
implies that $g=\pm f$?
\end{question}

Of course, one may replace the condition $|f(\lambda)|=|g(\lambda)|$ with $f(\lambda)^2=g(\lambda)^2$.
The difficulty of this question is that $f^2,g^2$ do not seem to belong to a variable bandwidth space
(there is no convolution theorem here). It is thus difficult to establish a sampling
theorem for those functions.

\smallskip

One conclusion of our research is that the variable band-width spaces introduced by Gr\"ochenig and Klotz
provide a deep theoretical family of spaces but with the drawback that practical questions in those spaces are hard 
to tackle. Recently Andreolli and Gr\"ochenig \cite{AG} introduced a new family of variable band-width spaces based on Wilson 
bases. Those spaces are easier to handle as their elements have a rather explicit description. Nevertheless the question of sign-retrieval does not seem simple in them and calls for further research.

\end{document}